\documentclass[oneside,english]{amsart}
\usepackage[T1]{fontenc}
\usepackage[latin9]{inputenc}
\usepackage{verbatim}
\usepackage{amsthm}
\usepackage{amstext}
\usepackage{amssymb}

\makeatletter
\numberwithin{equation}{section}
\numberwithin{figure}{section}

\usepackage{babel}
\providecommand{\corollaryname}{Corollary}
  \providecommand{\definitionname}{Definition}
  \providecommand{\examplename}{Example}
  \providecommand{\lemmaname}{Lemma}
  \providecommand{\propositionname}{Proposition}
\providecommand{\theoremname}{Theorem}
\providecommand{\remarkname}{Remark}

\theoremstyle{plain}
\newtheorem{thm}{\protect\theoremname}[section]
  \theoremstyle{definition}
  \newtheorem{defn}[thm]{\protect\definitionname}
  \theoremstyle{plain}
  \newtheorem{prop}[thm]{\protect\propositionname}
  \theoremstyle{plain}
  \newtheorem{cor}[thm]{\protect\corollaryname}
  \theoremstyle{plain}
  \newtheorem{lem}[thm]{\protect\lemmaname}
  \theoremstyle{definition}
  \newtheorem{example}[thm]{\protect\examplename}
  \theoremstyle{remark}
  \newtheorem{rem}[thm]{\protect\remarkname}  
\allowdisplaybreaks

\author[S. ter Horst]{Sanne ter Horst} \address{S. ter Horst, Department of Mathematics, Unit for BMI, North-West University, Potchefstroom, 2531, South Africa} \email{Sanne.TerHorst@nwu.ac.za}

\author[M. Messerschmidt]{Miek Messerschmidt} \address{M. Messerschmidt, Department of Mathematics, Unit for BMI, North-West University, Potchefstroom, 2531, South Africa} \email{mmesserschmidt@gmail.com}

\author[A.C.M. Ran]{Andr\'e C.M. Ran} \address{A.C.M. Ran, Department of Mathematics, FEW, VU university Amsterdam, De Boelelaan 1081a, 1081 HV Amsterdam, The Netherlands, and Unit for BMI, North-West~University, Potchefstroom,  2531, South Africa} \email{a.c.m.ran@vu.nl}

\subjclass[2010]{Primary: 47A05, 47B10 Secondary: 47L20, 46B03}
\keywords{Equivalence after extension, compact Banach space operators, $s$-numbers, operator ideals}

\begin{document}

\global\long\def\norm#1{\left\Vert #1\right\Vert }
{} \global\long\def\parenth#1{\left(\vphantom{#1}\right.\!\!#1\!\!\left.\vphantom{#1}\right)}
 \global\long\def\curly#1{\left\{  #1\right\}  }
 \global\long\def\set#1#2{\left\{  \vphantom{#1\ \vrule\ #2}\right.\!\!#1\ \vrule\ \linebreak[3]#2\!\!\left.\vphantom{#1\ \vrule\ #2}\right\}  }

 
\global\long\def\abs#1{\left\vert #1\right\vert }
 \global\long\def\span{\textup{span}}
 %
\global\long\def\R{\mathbb{R}}
{} \global\long\def\Z{\mathbb{Z}}
 \global\long\def\N{\mathbb{N}}
 \global\long\def\F{\mathbb{F}}
 \global\long\def\C{\mathbb{C}}
 \global\long\def\closedball#1{\textup{\textbf{B}}_{#1}}
 \global\long\def\operatorideal#1{I_{#1}}
 \global\long\def\rank{\textup{rank}}




\global\long\def\id{\textup{id}}



\global\long\def\range{\textup{ran}}




\global\long\def\smb#1{\left[\begin{smallmatrix}#1\end{smallmatrix}\right]}

\title[Equivalence after extension for compact operators]{Equivalence after extension for compact operators on Banach spaces}

\begin{abstract}
In recent years the coincidence of the operator relations equivalence
after extension and Schur coupling was settled for the Hilbert space
case, by showing that equivalence after extension implies equivalence
after one-sided extension. In the present paper we investigate consequences
of equivalence after extension for compact Banach space operators.
We show that generating the same operator ideal is necessary but not
sufficient for two compact operators to be equivalent after extension.
In analogy with the necessary and sufficient conditions for compact
Hilbert space operators to be equivalent after extension, in terms
of their singular values, we prove, under certain additional conditions,
the necessity of a similar relationship between the $s$-numbers of
two compact Banach space operators that are equivalent after extension,
for arbitrary $s$-functions.

We investigate equivalence after extension for operators on $\ell^{p}$-spaces.
We show that two operators that act on different $\ell^{p}$-spaces
cannot be equivalent after one-sided extension. Such operators can
still be equivalent after extension, for instance all invertible operators
are equivalent after extension, however, if one of the two operators
is compact, then they cannot be equivalent after extension. This contrasts
the Hilbert space case where equivalence after one-sided extension
and equivalence after extension are, in fact, identical relations.

Finally, for general Banach spaces $X$ and $Y$, we investigate consequences
of an operator on $X$ being equivalent after extension to a compact
operator on $Y$. We show that, in this case, a closed finite codimensional
subspace of $Y$ must embed into $X$, and that certain general Banach
space properties must transfer from $X$ to $Y$. We also show that
no operator on $X$ can be equivalent after extension to an operator
on $Y$, if $X$ and $Y$ are essentially incomparable Banach spaces. 
\end{abstract}

\maketitle

\section{Introduction}

Equivalence after extension (EAE) is an equivalence relation on bounded
Banach space operators that first appeared in the study of integral
equations \cite{BartGohbergKaashoek}; see Definition \ref{def:operatorrels}
below for its formal definition, as well as the definitions of the
other operator relations discussed in this paragraph. Part of the
advances made after the introduction of this notion came from the
observation that it coincided with another equivalence relation referred
to as matricial coupling (MC); in \cite{BartGohbergKaashoek} only
the implication (MC) $\Rightarrow$ (EAE) is used, and proved, while
the reverse implication (EAE) $\Rightarrow$ (MC) was settled in \cite{BartTsekanovskiMatricialCoupling}.
A few years later the operator relations again appeared, when in \cite{BartTsekanovski}
it was shown that a third operator relation, named Schur coupling
(SC), implies equivalence after extension and matricial coupling,
and the question was posed whether these three operator relations
coincide, i.e., if (EAE) $=$ (MC) $\Rightarrow$ (SC). Apart from
an affirmative answer in the case of Fredholm operators of index 0
(and without the index constraint in the case of Hilbert space operators),
little progress was made until recently. In \cite{TerHorstRan} the
three operator relations were shown to coincide for the classes of
Hilbert space operators with closed range and Banach space operators
that can be approximated in norm by an invertible operator, leading
to an affirmative answer in the case of Hilbert space operators on
separable Hilbert spaces. The general Hilbert space case was settled
by Timotin in \cite{Timotin} by showing that equivalence after extension
implies another operator relation, namely equivalence after one-sided
extension (EAOE), which was shown to imply Schur coupling in \cite{BartGohbergKaashoekRan}.
Specifically in the case of compact Hilbert space operators, a characterization
for two compact operators to be equivalent after extension is presented
by Timotin in \cite{Timotin} in terms of their singular values (cf.
Theorem \ref{thm:TimotinForCompactOperators} below).

In the current paper we focus on the notions of equivalence after
extension and equivalence after one-sided extension for compact Banach
space operators.

In the sequel the term `operator' will be short for bounded linear
operator and invertibility of an operator will imply that the inverse
is a (bounded) operator as well. All Banach spaces are assumed to
be over $\C$ and for given Banach spaces $X$ and $Y$ we write $B(X,Y)$
for the space of bounded linear operators from $X$ to $Y$, abbreviated
to $B(X)$ in case $X=Y$. By $X\oplus Y$ we denote the $\ell^{2}$-direct
sum of Banach spaces $X$ and $Y$. The identity operator on a Banach
space $X$ will be denoted by $\id_{X}$. With these definitions out
of the way, we are ready to formulate the operator relations discussed
in the first paragraph. 
\begin{defn}
\label{def:operatorrels} Let $T\in B(X)$ and $S\in B(Y)$ be Banach
space operators. 
\begin{enumerate}
\item We will say that $T$ and $S$ are \emph{equivalent after extension}
if there exist Banach spaces $X'$ and $Y'$ and invertible operators
invertible operators $E\in B(Y\oplus Y',X\oplus X')$ and $F\in B(X\oplus X',Y\oplus Y')$
such that 
\[
\left[\begin{array}{cc}
T & 0\\
0 & \id_{X'}
\end{array}\right]=E\left[\begin{array}{cc}
S & 0\\
0 & \id_{Y'}
\end{array}\right]F.
\]

\item We will say that $T$ and $S$ are \emph{equivalent after one-sided
extension} if $T$ and $S$ are equivalent after extension and one
of the Banach spaces $X'$ or $Y'$ can be chosen as the trivial Banach
space $\{0\}$. 
\item We will say that $T$ and $S$ are \emph{matricially coupled} if there
exists an invertible operator $U=\smb{U_{11}&U_{12}\\U_{21}&U_{22}}\in B(X \oplus Y, X\oplus Y)$
with inverse $V=\smb{V_{11}&V_{12}\\V_{21}&V_{22}}\in B(X \oplus Y, X\oplus Y)$,
so that $T=U_{11}$ and $S=V_{22}$. 
\item We will say that $T$ and $S$ are \emph{Schur coupled} if there exists
an operator $\smb{A&B\\C&D}\in B(X \oplus Y, X\oplus Y)$ such that
$A\in B(X)$ and $D\in B(Y)$ are invertible and $T=A-BD^{-1}C$ and
$S=D-CA^{-1}B.$
\end{enumerate}
\end{defn}
\begin{rem}
\label{rem:elem-eae-properties}If $T$ and $S$ are equivalent after
extension then the spaces $X'$ and $Y'$ in (1) above can always
be chosen so that $X'=Y$ and $Y'=X$ (cf. \cite[Lemma 4.1]{TerHorstRan}).
Throughout the rest of this paper this will be tacitly assumed. 

We also note, in the case that both $T$ and $S$ are invertible,
then an elementary construction will show that they are equivalent
after extension.
\end{rem}
Timotin's solution in \cite{Timotin}, that equivalence after extension
of Hilbert space operators implies their equivalence after one-sided
extension, relies first of all on the spectral theorem (after reducing
the general situation to that of positive operators without loss of
generality). This technique is, of course, not available for Banach
space operators without restricting to smaller classes of operators.
Secondly, Timotin's argument relies in an essential way on the identical
geometry that all Hilbert spaces share, in particular, that any Hilbert
space can be embedded into any other of greater dimension. This is
not possible for Banach spaces in general. As we shall see, the geometries
of the underlying spaces play a crucial role in the possibility of
two operators to be equivalent after (one-sided) extension. In fact,
Corollary \ref{cor:eae-does-not-imply-eaoe} will show that Timotin's
result, that equivalence after extension implies equivalence after
one-sided extension, does \emph{not} generalize to general Banach
spaces. 

Despite the lack of consistent geometrical structure across different
Banach spaces, some positive results are achievable. Compact Banach
space operators that are equivalent after extension are shown in Theorem
\ref{thm:EAE-Implies-Same-Ideal} to necessarily generate the same
ideal. This property is however not sufficient to imply equivalence
after extension as illustrated by Example~\ref{exa:generating-same-ideal-not-sufficient-for-eae}.

Using the general theory of $s$-numbers for Banach space operators
as replacement for Hilbert space operators' singular values, one is
still able to regain the necessity of certain relationships between
the $s$-numbers of compact Banach space operators that are equivalent
after extension or generate the same ideal, much akin to similar results
by Timotin and Schatten, cf. Theorem \ref{thm:TimotinForCompactOperators}
and Proposition \ref{lem:s-numbers-consequence-of-ran-lemma}, and
Theorem \ref{thm:SchattenIdealCharacterizationSpecialized} and Proposition
\ref{prop:Schatten-IT=00003D00003D00003DTS-imply-s-number-relation}
below. Example \ref{exa:n-1_on_lp_and_lq_not_eae} however shows that
Theorem \ref{thm:TimotinForCompactOperators} does not fully carry
over to Banach spaces in general.

For two Banach space operators with one of the operators compact,
them being equivalent after extension has far-reaching consequences
for the geometry of the underlying Banach spaces. An elementary application
of the Pitt-Rosenthal Theorem shows that the geometries of different
$\ell^{p}$-spaces are such that no compact operator on an $\ell^{p}$-space
can be equivalent after extension to any operator on a different $\ell^{p}$-space
(cf.~Proposition \ref{prop:no-compacts-on-different-lps-are-eae}).
We can go even further, by showing that no compact operator on a Banach
space $Y$ can be equivalent after extension to any operator on a
Banach space $X$, if the spaces $X$ and $Y$ are essentially incomparable
Banach spaces (cf.~Theorem \ref{thm:essentially-incomparable-no-operator-eae-to-a-compact}).

If a compact operator on a Banach space $Y$ is equivalent after extension
to any operator on a Banach space $X$, then a finite codimensional
subspace of $Y$ \emph{must} embed into $X$ (cf.~Theorem \ref{prop:Ran_finite_codim_subspace_embeds}).
The salient point of this result is that, for a Banach space operator
to be equivalent after extension to a compact Banach space operator,
the underlying Banach spaces' geometries \emph{must} be ``compatible
enough'' to allow for such an embedding. In fact, any Banach space
property that $X$ may have, that is also transferred to its closed
subspaces, and preserved under taking direct sums with finite dimensional
spaces, transfers from $X$ to $Y$ (cf.~Proposition \ref{prop:eae-to-compact-implies-transferrence-of-properties}).

\medskip{}
 We briefly describe the structure of the paper.

In Section \ref{sec:Operator-ideals-generated}, we will prove one
of our main results, Theorem \ref{thm:EAE-Implies-Same-Ideal}: That
compact Banach space operators that are equivalent after extension,
necessarily generate the same (operator) ideal. The proof relies on
Proposition \ref{prop:RanLemma}, which establishes what may be termed
a ``finite rank perturbed conjugation relationship'' that exists
between all compact operators on Banach spaces that are equivalent
after extension. In providing Example \ref{exa:generating-same-ideal-not-sufficient-for-eae},
we show that generating the same (operator) ideal is not sufficient
for compact Banach space operators to be equivalent after extension.

In Section \ref{sec:General-s-number-relationships-for-compact-eae}
we investigate $s$-number relationships for compact Banach space
operators. In Proposition \ref{prop:Schatten-IT=00003D00003D00003DTS-imply-s-number-relation},
we prove one direction of Schatten's characterization for Hilbert
space compact operators generating the same ideal in terms of their
singular values. Proposition \ref{lem:s-numbers-consequence-of-ran-lemma}
establishes the necessity of a relationship between the $s$-numbers
for compact Banach space operators that are equivalent after extension.
This result is analogous to Timotin's characterization in terms of
the singular values for compact Hilbert space operators that are equivalent
after extension (Theorem \ref{thm:TimotinForCompactOperators}). However,
Example \ref{exa:n-1_on_lp_and_lq_not_eae} will show that the full
characterization does not carry over to Banach spaces in general.

We investigate equivalence after extension for operators on $\ell^{p}$-spaces
in Section~\ref{sec:Equivalence-after-extension-for-lp}. The Pitt-Rosenthal
Theorem (Theorem \ref{thm:Pitt-Rosenthal-Theorem}) plays a crucial
role in our results. Employing this theorem, we show in Proposition
\ref{prop:no-lp-operators-are-eaoe} that no operators on different
$\ell^{p}$-spaces can ever be equivalent after one-sided extension.
This immediately establishes the existence of very simple operators
on $\ell^{p}$ that are equivalent after extension, but are not equivalent
after one-sided extension, cf. Corollary~\ref{cor:eae-does-not-imply-eaoe}.
This shows that Timotin's result, Theorem \ref{thm:Timotin-eae-equivalent-to-eaoe},
does not generalize to Banach spaces. We conclude the section by showing
in Proposition \ref{prop:no-compacts-on-different-lps-are-eae} that
no operator on an $\ell^{p}$-space can be equivalent after extension
to a compact operator on a different $\ell^{p}$-space.

Finally, in Section \ref{sec:Equivalence-after-extension-general},
we investigate some of the consequences of a Banach space operator
being equivalent after extension to a compact Banach space operator.
For Banach spaces $X$ and $Y$, we prove in Theorem \ref{thm:essentially-incomparable-no-operator-eae-to-a-compact},
that if $X$ and $Y$ are essentially incomparable (cf. Definition~\ref{def:incomparability}),
then no operator on $X$ can be equivalent after extension to a compact
operator on $Y$. On the other hand, if an operator on $X$ is equivalent
after extension to a compact operator on $Y$, Theorem \ref{prop:Ran_finite_codim_subspace_embeds}
shows that a closed finite codimensional subspace of $Y$ \emph{must
}embed into $X$. Also, Proposition \ref{prop:eae-to-compact-implies-transferrence-of-properties}
shows that any Banach space property that is transferred to closed
subspaces and preserved under the taking of direct sums with finite
dimensional spaces, transfers from $X$ to $Y$. Finally, Corollary
\ref{cor:properties-that-transfer} gives some specific examples of
such properties.

\medskip{}




\section{Operator ideals generated by compact operators and equivalence after
extension}

\label{sec:Operator-ideals-generated}

This section will establish that two compact operators on Banach spaces
that are equivalent after extension must necessarily generate the
same (operator) ideal. The converse implication is not true in general,
not even in the Hilbert space case, as illustrated in Example \ref{exa:generating-same-ideal-not-sufficient-for-eae}
below. 
\begin{defn}
Let $T\in B(X,Y)$ be a Banach space operator. For any Banach spaces
$Z_{1}$ and $Z_{2}$, we define 
\[
\operatorideal T(Z_{1},Z_{2}):=\bigcup_{n\in\N}\set{\sum_{j=1}^{n}R_{j}TR_{j}'}{R_{j}\in B(Y,Z_{2}),\ R_{j}'\in B(Z_{1},X)}.
\]
By $\operatorideal T$ we will denote the (proper) class $\bigcup_{Z_{1},Z_{2}}\operatorideal T(Z_{1},Z_{2})$,
and refer to $\operatorideal T$ as the \emph{operator ideal generated
by }$T$. 
\end{defn}
It is easy to see that $\operatorideal T$ is, in fact, an operator
ideal in the sense of Pietsch \cite[Chapter 1]{PietschOperatorIdeals},
provided $T\neq0$. In this case we note that $I_{T}$ also contains
all finite rank operators.

The following proposition will be a crucial ingredient in the current
and following section. It establishes what may be termed ``a finite
rank perturbed conjugation relationship'' that exists between compact
operators on Banach spaces that are equivalent after extension.

We note that the symmetry of equivalence after extension allows us
to exchange the roles of $T$ and $S$ in the following result without
any loss of generality. 
\begin{prop}
\label{prop:RanLemma} Let $T\in B(X)$ and $S\in B(Y)$ be compact
Banach space operators that are equivalent after extension. Then there
exist operators $G\in B(Y,X)$, $H\in B(X,Y)$ and a finite rank operator
$R\in B(X)$ such that $T=GSH+R$. \end{prop}
\begin{proof}
Since $T$ and $S$ are equivalent after extension, there exist invertible
operators $E\in B(Y\oplus X,X\oplus Y)$ and $F\in B(X\oplus Y,Y\oplus X)$
satisfying 
\[
\left[\begin{array}{cc}
T & 0\\
0 & \id_{Y}
\end{array}\right]=E\left[\begin{array}{cc}
S & 0\\
0 & \id_{X}
\end{array}\right]F.
\]
Furthermore, by \cite[Theorem 2.1]{TerHorstRan}, we may choose operators
$G_{11},G_{21},G_{22},H_{11},H_{21}$ and $H_{22}$ in such a way
that 
\[
E=\left[\begin{array}{cc}
G_{11} & T\\
G_{21} & G_{22}
\end{array}\right]\quad\textup{and}\quad F=\left[\begin{array}{cc}
H_{11} & \id_{Y}\\
H_{21}T & H_{22}
\end{array}\right].
\]
From %
\begin{eqnarray*}
\left[\begin{array}{cc}
T & 0\\
0 & \id_{Y}
\end{array}\right] & = & \left[\begin{array}{cc}
G_{11} & T\\
G_{21} & G_{22}
\end{array}\right]\left[\begin{array}{cc}
S & 0\\
0 & \id_{X}
\end{array}\right]\left[\begin{array}{cc}
H_{11} & \id_{Y}\\
H_{21}T & H_{22}
\end{array}\right]\\
 & = & \left[\begin{array}{cc}
G_{11}SH_{11} & G_{11}S\\
G_{21}SH_{11} & G_{21}S
\end{array}\right]+\left[\begin{array}{cc}
TH_{21}T & TH_{22}\\
G_{22}H_{21}T & G_{22}H_{22}
\end{array}\right],
\end{eqnarray*}
we notice that $T=G_{11}SH_{11}+TH_{21}T$, which we may rearrange
to 
\[
T(\id_{X}-H_{21}T)=G_{11}SH_{11}.
\]
The operator $\id_{X}-H_{21}T$ is Fredholm, and of index zero, cf.
\cite[Corollary XI.4.3]{GohbergGoldbergKaashoekLinearClassesI}. Hence
there exists an invertible operator $L\in B(X)$ and finite rank operator
$K\in B(X)$ satisfying $\id_{X}-H_{21}T=L+K$, cf. \cite[Theorem XI.5.3]{GohbergGoldbergKaashoekLinearClassesI}.
Define $G:=G_{11}$, $H:=H_{11}L^{-1}$ and $R:=-TKL^{-1}$. Then
\[
T=G_{11}SH_{11}L^{-1}-TKL^{-1}=GSH+R
\]
and $R$ has finite rank. 
\end{proof}
Before proving Theorem \ref{thm:EAE-Implies-Same-Ideal}, our main
result in this section, we give a number of consequences of the previous
result that we will need in later sections.

The following corollary is an immediate consequence of Proposition
\ref{prop:RanLemma}: 
\begin{cor}
\label{cor:finrank} Let $T\in B(X)$ and $S\in B(Y)$ be compact
Banach space operators that are equivalent after extension. The operator
$T$ has finite rank if and only if $S$ has finite rank. 
\end{cor}
We remark that in the previous proposition and corollary compactness
of \emph{both} operators is required. Equivalence after extension
of a finite rank operator with a second operator does not imply that
the second operator has finite rank. In fact, in many of the original
examples (cf.~\cite{BartGohbergKaashoek}) it is shown that an integral
operator is equivalent after extension (or rather, matricially coupled)
to an operator on a finite dimensional space from which, amongst others,
it can be concluded that the integral operator is Fredholm.

We briefly give an alternative proof of the previous corollary using
Lemma \ref{lem:closed_range_preserved_under_eae} which may be of
independent interest. 
\begin{lem}
\label{lem:closed_range_preserved_under_eae} Let $T\in B(X)$ and
$S\in B(Y)$ be Banach space operators that are equivalent after extension.
The operator $T$ has closed range if and only if $S$ has closed
range. \end{lem}
\begin{proof}
Let $T$ and $S$ be equivalent after extension. Then there exist
invertible operators $E:Y\oplus X\to X\oplus Y$ and $F:X\oplus Y\to Y\oplus X$,
so that $\smb{T&0\\0&\id_Y}=E\smb{S&0\\0&\id_X}F$. Elementary arguments
will establish that $T$ has closed range if and only if $\smb{T&0\\0&\id_Y}$
has closed range, and also that $S$ has closed range if and only
if $\smb{S&0\\0&\id_X}$ has closed range. Since $E$ and $F$ are
invertible, $\smb{T&0\\0&\id_Y}$ has closed range if and only if
$\smb{S&0\\0&\id_X}$ has closed range. 
\end{proof}

\begin{proof}[Alternative proof of Corollary \ref{cor:finrank}]
Let $T\in B(X)$ and $S\in B(Y)$ be compact Banach space operators
that are equivalent after extension. Assume $T$ has finite rank.
Then $T$ has closed range, and hence, by Lemma \ref{lem:closed_range_preserved_under_eae},
$S$ also has closed range. Since $S$ is compact and has closed range
it must have finite rank. The converse follows similarly. 
\end{proof}

Using Proposition \ref{prop:RanLemma}, our main result in this section
becomes a matter of routine: 
\begin{thm}
\label{thm:EAE-Implies-Same-Ideal} Let $T\in B(X)$ and $S\in B(Y)$
be non-zero compact Banach space operators. If $T$ and $S$ are equivalent
after extension, then $I_{T}=I_{S}$. \end{thm}
\begin{proof}
By Proposition \ref{prop:RanLemma}, there exist operators $G,H,G',H'$
and finite rank operators $R$ and $R'$ such that $T=GSH+R$ and
$S=G'TH'+R'$. We note that $GSH\in\operatorideal S$, and also, since
$R$ is of finite rank, that $R\in\operatorideal S$. We conclude
that $T=GSH+R\in\operatorideal S$, and hence $\operatorideal T\subseteq\operatorideal S$.
Similarly, $\operatorideal S\subseteq\operatorideal T$, and hence,
$\operatorideal T=\operatorideal S$. 
\end{proof}
The converse of Theorem \ref{thm:EAE-Implies-Same-Ideal} is false.
We will briefly elaborate on this claim.

In \cite{Schatten} Schatten characterized ideals of compact operators
on Hilbert spaces in terms of the properties of their singular values.
For ideals generated by single compact operators \cite[Theorem 12]{Schatten}
specializes to the following: 
\begin{thm}
\label{thm:SchattenIdealCharacterizationSpecialized}Let $T$ and
$S$ be compact operators on a Hilbert space $H$ and let $\{t_{n}\}$
and $\{s_{n}\}$ denote their respective sequences of singular values.
The following are equivalent: 
\begin{enumerate}
\item The operators $T$ and $S$ generate the same ideal in $B(H)$. 
\item There exist constants $M>0$ and $m\in\N$ such that both 
\[
t_{m(n-1)+j}\leq Ms_{n}\quad\text{and}\quad s_{m(n-1)+j}\leq Mt_{n},
\]
hold for all $n\in\N$, and $j\in\{1,\ldots,m-1\}$. 
\end{enumerate}
\end{thm}
With the previous result, one can easily find examples of compact
operators on $\ell^{2}$ that are not equivalent after extension by
finding compact operators that do not generate the same ideal. E.g.,
the compact diagonal operators $[n^{-1}]$ and $[2^{-n}]$ on $\ell^{2}$
are not equivalent after extension, where, for any bounded sequence
$\{a_{n}\}_{n\in\N}\subseteq\C$, by $[a_{n}]\in B(\ell^{2})$ we
denote the diagonal operator $[a_{n}]:(x_{1},x_{2},\ldots)\mapsto(a_{1}x_{1},a_{2}x_{2},\ldots)$
with $(x_{1},x_{2},\ldots)\in\ell^{2}$.

In \cite{Timotin} Timotin established the following characterization
connecting the equivalence after extension of compact operators on
Hilbert spaces to their singular values satisfying a specific relationship
\cite[Theorem 6.3]{Timotin}. 
\begin{thm}
\label{thm:TimotinForCompactOperators} Let $T$ and $S$ be compact
operators on Hilbert spaces and let $\{t_{n}\}$ and $\{s_{n}\}$
denote their respective sequences of singular values. The following
are equivalent: 
\begin{enumerate}
\item The operators $T$ and $S$ are equivalent after extension. 
\item There exist constants $\delta\in(0,1)$ and $m\in\N$, such that either
\[
\delta\leq\frac{s_{n}}{t_{n+m}}\leq\delta^{-1}\quad\text{or}\quad\delta\leq\frac{t_{n}}{s_{n+m}}\leq\delta^{-1}
\]
holds for all $n\in\N$. 
\end{enumerate}
\end{thm}
Using this result, we can now show that the converse of Theorem \ref{thm:EAE-Implies-Same-Ideal}
is false:
\begin{example}
\label{exa:generating-same-ideal-not-sufficient-for-eae}Consider
the two compact diagonal operators $[2^{-n}]$ and $[2^{-2n}]$ on
$\ell^{2}$. We note that, for any $m,n,j\in\N$, since $2^{-2mn}\leq2^{-n}$
and $2^{-2j}\leq1$, that 
\begin{eqnarray*}
2^{-2(m(n-1)+j)} & = & 2^{2m}2^{-2mn}2^{-2j}\\
 & \leq & 2^{2m}2^{-n}.
\end{eqnarray*}
Similarly, if, in addition $m\geq2$, then $2^{-mn}\leq2^{-2n}$ and
$2^{m}\leq2^{2m}$, so that 
\begin{eqnarray*}
2^{-(m(n-1)+j)} & = & 2^{m}2^{-mn}2^{-j}\\
 & \leq & 2^{m}2^{-2n}\\
 & \leq & 2^{2m}2^{-2n}.
\end{eqnarray*}
Therefore, by Theorem \ref{thm:SchattenIdealCharacterizationSpecialized}
(therein taking $m:=2$ and $M:=2^{4}=16$), $[2^{-n}]$ and $[2^{-2n}]$
generate the same ideal on $\ell^{2}$. Hence we conclude that $\operatorideal{[2^{-n}]}=\operatorideal{[2^{-2n}]}$.

On the other hand, for any fixed $m\in\N$, 
\[
\frac{2^{-n}}{2^{-2(n+m)}}=2^{n+2m}\to\infty
\]
and 
\[
\frac{2^{-2n}}{2^{-(n+m)}}=\frac{1}{2^{n-m}}\to0,
\]
as $n\to\infty$. Hence, by Theorem \ref{thm:TimotinForCompactOperators},
the operators $[2^{-n}]$ and $[2^{-2n}]$ on $\ell^{2}$ are not
equivalent after extension. 
\end{example}

\section{General $s$-number relationships of compact operators that are equivalent
after extension\label{sec:General-s-number-relationships-for-compact-eae}}

In this section we investigate the possibilities of extending the
Hilbert space case results of Theorems \ref{thm:SchattenIdealCharacterizationSpecialized}
and \ref{thm:TimotinForCompactOperators} to the Banach space setting.
For both results we only prove (parts of) the implication (1) $\Rightarrow$
(2), where the role of the singular values are now played by $s$-numbers. 

The proof of the reverse implication (2) $\Rightarrow$ (1) in Theorem
\ref{thm:TimotinForCompactOperators} given in \cite{Timotin}, relies
heavily on the fact that the operator relations equivalence after
extension and equivalence after one-sided extension coincide, which
is not the case in the general Banach space setting, as will be shown
in the next section. Example \ref{exa:n-1_on_lp_and_lq_not_eae} gives
an explicit example where the implication (2) $\Rightarrow$ (1) from
Theorem \ref{thm:TimotinForCompactOperators} fails for compact Banach
space operators, with specific choices of $s$-numbers playing the
role of the operators' singular values.

We begin by defining $s$-functions and $s$-numbers which play the
same role as singular values of operators on Hilbert spaces. For a
more complete treatment of these objects we refer the reader to \cite{PietschOperatorIdeals,PietschEigenvaluesAndSNumbers}. 
\begin{defn}
By an\emph{ $s$-function}\footnote{Pietsch's axioms for $s$-functions across \cite{PietschEigenvaluesAndSNumbers,PietschOperatorIdeals}
are different. What we call an $s$-function, Pietsch calls an \emph{additive}
$s$-function in \cite{PietschOperatorIdeals}, and an $s$-scale
in \cite{PietschEigenvaluesAndSNumbers}.}\emph{ }we will mean a rule for assigning to any operator $T\in B(X,Y)$
for any Banach spaces $X$ and $Y$, a sequence of numbers $\{s_{n}(T)\}$,
\emph{the sequence of $s$-numbers of $T$}, satisfying the following
conditions: 
\begin{enumerate}
\item For every $T\in B(X,Y)$, $\norm T=s_{1}(T)\geq s_{2}(T)\geq\ldots\geq0$. 
\item For every $m,n\in\N$ and $S,T\in B(X,Y)$, $s_{n+m-1}(S+T)\leq s_{n}(S)+s_{m}(T).$ 
\item For $T\in B(X,Y)$, $S\in B(Y,Z_{2})$ and $R\in B(Z_{1},X)$, with
$Z_{1}$ and $Z_{2}$ arbitrary Banach spaces, and for every $n\in\N$,
$s_{n}(STR)\leq\norm S\norm Rs_{n}(T).$ 
\item If $T\in B(X,Y)$ and $\rank T<n$, then $s_{n}(T)=0$. 
\item For all $n\in\N$, $s_{n}(\id_{\ell_{n}^{2}})=1$, where $\ell_{n}^{2}$
denotes $\mathbb{C}^{n}$ with the $\ell^{2}$-norm.
\end{enumerate}
\end{defn}
Since in the case of Hilbert space operators all $s$-functions coincide
(with the singular values) \cite[Theorem 11.3.4]{PietschOperatorIdeals},
the following result generalizes the necessity of (2) in Theorem \ref{thm:SchattenIdealCharacterizationSpecialized}. 
\begin{prop}
\label{prop:Schatten-IT=00003D00003D00003DTS-imply-s-number-relation}
Let $T\in B(X)$ and $S\in B(Y)$ be Banach spaces operators. If $\operatorideal T=\operatorideal S$,
then there exist constants $M>0$ and $m\in\N$ such that, for any
$s$-function $s$, 
\[
s_{m(n-1)+j}(T)\leq Ms_{n}(S)\quad\text{and}\quad s_{m(n-1)+j}(S)\leq Ms_{n}(T)\quad(n\in\N)
\]
for all $j\in\{0,\ldots,m-1\}$. \end{prop}
\begin{proof}
If $\operatorideal T=\operatorideal S$, then there exists a constant
$m$ and operators $R_{j},R_{j}'''\in B(X,Y)$ and $R_{j}',R_{j}''\in B(Y,X)$
(by choosing some to be zero, if need be) for $j\in\{1,\ldots,m\}$,
such that $S=\sum_{j=1}^{m}R_{j}TR_{j}'$ and $T=\sum_{j=1}^{m}R_{j}''SR_{j}'''$.
For all $n\in\N$, by the properties of $s$-functions, we have

\begin{eqnarray*}
s_{m(n-1)}(S) & = & s_{mn-m}(S)=s_{mn-m}\parenth{\sum_{j=1}^{m}R_{j}TR_{j}'}\\
 & = & s_{n+(m-1)n-(m-1)-1}\parenth{\sum_{j=1}^{m}R_{j}TR_{j}'}\\
 & \leq & s_{n}\parenth{R_{1}TR_{1}'}+s_{(m-1)n-(m-1)}\parenth{\sum_{j=2}^{m}R_{j}TR_{j}'}\\
 & \ldots\\
 & \leq & \sum_{j=1}^{m}s_{n}(R_{j}TR_{j}')\leq\parenth{\sum_{j=1}^{m}\norm{R_{j}}\norm{R_{j}'}}s_{n}(T).
\end{eqnarray*}
Similarly, we obtain $s_{m(n-1)}(T)\leq\parenth{\sum_{j=1}^{m}\norm{R_{j}''}\norm{R_{j}'''}}s_{n}(S)$
for all $n\in\N$. Taking $M:=\max\curly{\sum_{j=1}^{m}\norm{R_{j}}\norm{R_{j}'},\sum_{j=1}^{m}\norm{R_{j}''}\norm{R_{j}'''}}$,
and since $s$-numbers are decreasing, we obtain 
\[
s_{m(n-1)+j}(T)\leq s_{m(n-1)}(T)\leq Ms_{n}(S)
\]
and 
\[
s_{m(n-1)+j}(S)\leq s_{m(n-1)}(S)\leq Ms_{n}(T),
\]
for all $n\in\N$ and $j\in\{0,\ldots,m-1\}.$ 
\end{proof}
The following lemma shows that even in the Banach space case one of
the inequalities in (2) of Theorem \ref{thm:TimotinForCompactOperators}
is still implied by equivalence after extension, without any additional
assumptions. 
\begin{prop}
\label{lem:s-numbers-consequence-of-ran-lemma} Let $s$ be any $s$-function.
Let $T\in B(X)$ and $S\in B(Y)$ be compact Banach space operators
that are equivalent after extension. Let $G\in B(Y,X)$, $H\in B(X,Y)$
and $R\in B(X)$ such that $T=GSH+R$, where $R$ is of finite rank
(cf. Proposition \textup{\ref{prop:RanLemma}}). Then, for every $n\in\N$
and $m\geq\rank\,R$, 
\[
s_{n+m}(T)\leq\|G\|\|H\|s_{n}(S).
\]
\end{prop}
\begin{proof}
By the properties of $s$-functions we have 
\begin{eqnarray*}
s_{n+m}(T) & \leq & s_{n+\rank R}(T)=s_{n+(\rank R+1)-1}(T)\\
 & \leq & s_{n+(\rank R+1)-1}(GSH+R)\\
 & \leq & s_{n}(GSH)+s_{\rank R+1}(R)\leq\norm G\norm Hs_{n}(S).
\end{eqnarray*}

\end{proof}

\begin{rem}
The previous result shows that elementary arguments will establish
one direction of the analogous inequalities from Theorem \ref{thm:TimotinForCompactOperators}(2)
for compact Banach space operators that are equivalent after extension.
One can also obtain the reverse inequality if one were to assume,
under the hypothesis of Proposition \ref{lem:s-numbers-consequence-of-ran-lemma},
that $s_{n}(S)=s_{n}(G'TH'+R')\leq Ms_{n+\max\{\rank R',\rank R\}}(G'TH')$
for all $n\in\N$, where $G',H',R'$ are as would be obtained from
Proposition \ref{prop:RanLemma}, with $R'$ finite rank, and satisfying
$S=G'TH'+R'$. This is a somewhat unnatural assumption to make, and
seems to not be easily verified for examples. We therefore omit any
formal treatment of it.
\end{rem}

\section{Equivalence after extension for operators on $\ell^{p}$-spaces}

\label{sec:Equivalence-after-extension-for-lp}

In this section we consider operators on different $\ell^{p}$-spaces.
Here, for $1\leq p\leq\infty$, by $\ell^{p}$ we will denote the
sequence space $\ell^{p}(\N)$. The results obtained here illustrate
that an ``incompatibility'' in the geometry of the underlying Banach
spaces on which operators act has consequences for whether certain
classes of operators can be equivalent after (one-sided) extension.

The fact that all Hilbert spaces have ``the same geometry'' allows
for the establishment of the following result, due to Timotin \cite[Theorem 5.4]{Timotin}: 
\begin{thm}
\label{thm:Timotin-eae-equivalent-to-eaoe} Let $T\in B(H_{1})$ and
$S\in B(H_{2})$ be Hilbert space operators. Then the following are
equivalent: 
\begin{enumerate}
\item The operators $T$ and $S$ are equivalent after extension. 
\item The operators $T$ and $S$ are equivalent after one-sided extension. 
\end{enumerate}
\end{thm}
This result does not carry over to the case of Banach space operators,
as we will see from the results and examples presented below, cf.
Corollary \ref{cor:eae-does-not-imply-eaoe}.

All results in this section hinge on The Pitt-Rosenthal Theorem \cite[Theorem~5.14]{Morrison}: 
\begin{thm}[The Pitt-Rosenthal Theorem]
\label{thm:Pitt-Rosenthal-Theorem} Any operator in $B(\ell^{p},\ell^{q})$,
where $1\leq q<p<\infty$, is compact. 
\end{thm}
Obviously, every Hilbert space can be isometrically embedded into
any other Hilbert space of higher dimension, i.e., all Hilbert spaces
essentially have the same geometry. This is not true in the general
Banach space case: The Pitt-Rosenthal Theorem even implies that no
infinite dimensional subspace of $\ell^{p}$ is topologically isomorphic
to a subspace of $\ell^{q}$ (and vice versa) when $1\leq p\neq q<\infty$,
\cite[Corollary 5.10]{Morrison}. 
\begin{prop}
\label{prop:no-lp-operators-are-eaoe} No operators $T\in B(\ell^{p})$
and $S\in B(\ell^{q})$ are ever equivalent after one-sided extension
whenever $1\leq p\neq q<\infty$. \end{prop}
\begin{proof}
Suppose $T\in B(\ell^{p})$ and $S\in B(\ell^{q})$ are equivalent
after one-sided extension, where $1\leq p\neq q<\infty$. Then (by
perhaps exchanging the roles of $T$ and $S$ if necessary) there
exists a Banach space $X$ and operators $A\in B(\ell^{q},\ell^{p})$,
$B\in B(X,\ell^{p})$, $C\in B(\ell^{p},\ell^{q})$ and $D\in B(\ell^{p},X)$
such that 
\[
\left[\begin{array}{cc}
A & B\end{array}\right]:\ell^{q}\oplus X\to\ell^{p}\quad\mbox{and}\quad\left[\begin{array}{c}
C\\
D
\end{array}\right]:\ell^{p}\to\ell^{q}\oplus X
\]
are invertible and 
\begin{eqnarray*}
T & = & \left[\begin{array}{cc}
A & B\end{array}\right]\left[\begin{array}{cc}
S & 0\\
0 & \id_{X}
\end{array}\right]\left[\begin{array}{c}
C\\
D
\end{array}\right].
\end{eqnarray*}
Let $G\in B(\ell^{p},\ell^{q})$ and $H\in B(\ell^{p},X)$ be such
that $\left[\begin{smallmatrix}G\\
H
\end{smallmatrix}\right]$ is the inverse of $\left[\begin{smallmatrix}A & B\end{smallmatrix}\right]$,
i.e., 
\[
\left[\begin{array}{cc}
\id_{\ell^{q}} & 0\\
0 & \id_{X}
\end{array}\right]=\left[\begin{array}{c}
G\\
H
\end{array}\right]\left[\begin{array}{cc}
A & B\end{array}\right]=\left[\begin{array}{cc}
GA & GB\\
HA & HB
\end{array}\right].
\]
By The Pitt-Rosenthal Theorem, either $A\in B(\ell^{q},\ell^{p})$
or $G\in B(\ell^{p},\ell^{q})$ is compact, so that $\id_{\ell^{q}}=GA$
is also compact, which is absurd. We conclude that $T$ and $S$ cannot
be equivalent after one-sided extension. \end{proof}
\begin{cor}
\label{cor:eae-does-not-imply-eaoe} Let $T\in B(\ell^{p})$ and $S\in B(\ell^{q})$
be invertible, where $1\leq q\neq p<\infty$. Then $T$ and $S$ are
equivalent after extension, but are not equivalent after one-sided
extension. \end{cor}
\begin{proof}
All invertible operators are equivalent after extension (cf. Remark
\ref{rem:elem-eae-properties}). The previous result shows that $T$
and $S$ cannot be equivalent after one-sided extension. 
\end{proof}
Although operators $T\in B(\ell^{p})$ and $S\in B(\ell^{q})$ can
still be equivalent after extension whenever $1\leq q\neq p<\infty$,
by another application of the Pitt-Rosenthal Theorem we will now show
that this cannot occur if one of the operators is compact. 
\begin{prop}
\label{prop:no-compacts-on-different-lps-are-eae} Let $T\in B(\ell^{p})$
and $S\in B(\ell^{q})$ with $1\leq q\neq p<\infty$. If either $T$
or $S$ is compact, then $T$ and $S$ cannot be equivalent after
extension. \end{prop}
\begin{proof}
By perhaps exchanging the roles of $S$ and $T$, we may assume that
$S$ is compact. The equivalence after extension of $T$ and $S$
implies that there exist invertible operators $F=\smb{F_{11}&F_{12}\\F_{21}&F_{22}}\in B(\ell^p\oplus\ell^q,\ell^q\oplus\ell^p)$
and $E=\smb{E_{11}&E_{12}\\E_{21}&E_{22}}\in B(\ell^q\oplus\ell^p,\ell^p\oplus\ell^q)$
such that 
\begin{eqnarray*}
\left[\begin{array}{cc}
T & 0\\
0 & \id_{\ell^{q}}
\end{array}\right] & =E\left[\begin{array}{cc}
S & 0\\
0 & \id_{\ell^{p}}
\end{array}\right]F=\left[\begin{array}{cc}
\ldots & \ldots\\
\ldots & E_{21}SF_{12}+E_{22}F_{22}
\end{array}\right].
\end{eqnarray*}
By The Pitt-Rosenthal Theorem, either $E_{22}\in B(\ell^{p},\ell^{q})$
or $F_{22}\in B(\ell^{q},\ell^{p})$ is compact. Therefore, since
$S$ is compact, $E_{21}SF_{12}+E_{22}F_{22}=\id_{\ell^{q}}$ is compact,
which is absurd. We conclude that $T$ and $S$ cannot be equivalent
after extension whenever $T$ or $S$ is compact. 
\end{proof}
A curious consequence of the previous corollary is that compact operators
with identical representations as matrices, but acting on different
$\ell^{p}$-spaces, cannot be equivalent after extension, as is shown
by Example \ref{exa:n-1_on_lp_and_lq_not_eae} below. This illustrates
the importance that the geometry of the underlying spaces play in
the possibility of certain classes of operators on them being equivalent
after extension. Furthermore, this example also shows that the implication
(2) $\Rightarrow$ (1) from Theorem \ref{thm:TimotinForCompactOperators}
does not carry over to Banach spaces in general with specific choices
of $s$-numbers standing in for singular values.

As before, for any bounded sequence $\{a_{n}\}_{n\in\N}\in\C$, by
$[a_{n}]\in B(\ell^{p})$ we will denote the diagonal operator $[a_{n}]:(x_{1},x_{2},\ldots)\mapsto(a_{1}x_{1},a_{2}x_{2},\ldots)$
with $(x_{1},x_{2},\ldots)\in\ell^{p}$.
\begin{example}
\label{exa:n-1_on_lp_and_lq_not_eae}Let $T$ and $S$ both denote
the compact diagonal operator $[n^{-1}]$ acting on $\ell^{p}$ and
$\ell^{q}$ respectively, with $1\leq p\neq q<\infty.$ Then $T$
and $S$ are not equivalent after extension by Proposition \ref{prop:no-compacts-on-different-lps-are-eae}.
By \cite[Theorem 11.11.3]{PietschOperatorIdeals} the approximation
numbers, Kolmogorov numbers and Gelfand numbers of both $T$ and $S$
are all three equal to the sequence $\{n^{-1}\}$. By taking any of
these three $s$-functions to play the role of the singular values
then shows that condition (2) in Theorem \ref{thm:TimotinForCompactOperators}
is satisfied, while $T$ and $S$ are not equivalent after extension. 
\end{example}

\section{Equivalence after extension for compact operators on general Banach
spaces\label{sec:Equivalence-after-extension-general}}

In the previous section we have seen that no compact operator on an
$\ell^{p}$-space can ever be equivalent after extension to an operator
on a different $\ell^{p}$-space. In this section we will prove similar
results for compact operators on more general Banach spaces. 
\begin{defn}
\label{def:incomparability}Let $X$ and $Y$ be Banach spaces. 
\begin{enumerate}
\item An operator $S\in B(X,Y)$ is called \emph{inessential} if, for all
operators $T\in B(Y,X)$, the operator $\id_{X}-TS$ is Fredholm.
The set of inessential operators in $B(X,Y)$ is denoted by $\mathcal{J}(X,Y)$. 
\item The Banach spaces $X$ and $Y$ are said to be \emph{essentially incomparable}
if $B(X,Y)=\mathcal{J}(X,Y)$. 
\item The Banach spaces $X$ and $Y$ are said to be \emph{totally incomparable}
if no infinite dimensional subspace of $X$ is topologically isomorphic
to a subspace of $Y$, and vice versa. 
\end{enumerate}
\end{defn}
Total incomparability was introduced by Rosenthal in \cite{Rosenthal}.
The notion of an inessential operator originated in \cite{PietschInessentialOperators}
and essentially incomparability was introduced in \cite{Gonzalez94}
(see \cite[Theorem 2]{Gonzalez94} for a characterization of pairs
of spaces that are essentially incomparable). We note that essential
incomparability of Banach spaces is symmetric, i.e., $B(X,Y)=\mathcal{J}(X,Y)$
if and only if $B(Y,X)=\mathcal{J}(Y,X)$, \cite[Proposition 1]{Gonzalez94}.
Furthermore, total incomparability implies essential incomparability,
but the converse is false (cf.~\cite{Gonzalez94}).

The Pitt-Rosenthal Theorem implies that different $\ell^{p}$-spaces
are totally incomparable \cite[Corollary 5.10]{Morrison}, and hence
essentially incomparable. The following result is therefore a generalization
of Proposition \ref{prop:no-compacts-on-different-lps-are-eae}, and,
together with \cite[Theorem 1]{Gonzalez94}, yields many more examples
of pairs of Banach spaces on which no operators on the one space can
be equivalent after extension to a compact operator on the other.
A more exotic example of a pair of essentially incomparable spaces
is any $C(K)$-space (which has the Dunford-Pettis property \cite[Theorem 5.4.5]{AlbiacKalton})
and the Tsirelson space (which is reflexive \cite[Theorem 10.3.2]{AlbiacKalton}),
which are then essentially incomparable by \cite[Theorem 1]{Gonzalez94}. 
\begin{thm}
\label{thm:essentially-incomparable-no-operator-eae-to-a-compact}Let
$X$ and $Y$ be infinite dimensional Banach spaces that are essentially
incomparable. Then no compact operator $S\in B(Y)$ is equivalent
after extension to any operator $T\in B(X)$. \end{thm}
\begin{proof}
Suppose $S\in B(Y)$ is compact and equivalent after extension to
$T\in B(X)$. Then there exist invertible operators $E=\smb{E_{11}&E_{12}\\E_{21}&E_{22}}\in B(Y\oplus X,X \oplus Y)$
and $F=\smb{F_{11}&F_{12}\\F_{21}&F_{22}}\in B(X\oplus Y,Y \oplus X)$,
so that 
\begin{eqnarray*}
\left[\begin{array}{cc}
T & 0\\
0 & \id_{Y}
\end{array}\right] & =E\left[\begin{array}{cc}
S & 0\\
0 & \id_{X}
\end{array}\right]F=\left[\begin{array}{cc}
\ldots & \ldots\\
\ldots & E_{21}SF_{12}+E_{22}F_{22}
\end{array}\right].
\end{eqnarray*}
Therefore $\id_{Y}=E_{21}SF_{12}+E_{22}F_{22}.$ Since $X$ and $Y$
are essentially incomparable, it follows that $\id_{Y}-E_{22}F_{22}$
is Fredholm. On the other hand, $\id_{Y}-E_{22}F_{22}=E_{21}SF_{12}$
is compact. That $\id_{Y}-E_{22}F_{22}$ is both Fredholm and compact,
implies that $Y$ is finite dimensional, contrary to our assumption.
We conclude that $T$ and $S$ cannot be equivalent after extension. 
\end{proof}
In our next result, we show that a consequence of an operator on some
Banach space $X$ being equivalent after extension to a compact operator
on another Banach space $Y$, is that a complemented subspace of finite
codimension in $Y$ must necessarily embed into $X$. In other words,
the geometry of $X$ \emph{must} be ``compatible enough'' with that
of $Y$ to allow for such an embedding.
\begin{thm}
\label{prop:Ran_finite_codim_subspace_embeds} Let $X$ and $Y$ be
Banach spaces and $T\in B(X)$ and $S\in B(Y)$ operators which are
equivalent after extension. If $S$ is compact, then there exists
a closed subspace of $Y$ of finite codimension that is topologically
isomorphic to a closed subspace of $X$. \end{thm}
\begin{proof}
Let $T\in B(X)$ and $S\in B(Y)$, with $S$ compact, be equivalent
after extension. Then there exist invertible operators $E=\smb{E_{11}&E_{12}\\E_{21}&E_{22}}\in B(Y\oplus X,X \oplus Y)$
and $F=\smb{F_{11}&F_{12}\\F_{21}&F_{22}}\in B(X\oplus Y,Y \oplus X)$
such that 
\begin{eqnarray*}
\left[\begin{array}{cc}
T & 0\\
0 & \id_{Y}
\end{array}\right] & =E\left[\begin{array}{cc}
S & 0\\
0 & \id_{X}
\end{array}\right]F=\left[\begin{array}{cc}
\ldots & \ldots\\
\ldots & E_{21}SF_{12}+E_{22}F_{22}
\end{array}\right].
\end{eqnarray*}
Since $E_{22}F_{22}=\id_{Y}-E_{21}SF_{12}$ and $S$ is compact, $E_{22}F_{22}$
is Fredholm, and therefore has finite dimensional kernel. In particular,
$F_{22}\in B(Y,X)$ has finite dimensional kernel. Since all finite
dimensional spaces are complemented, there exists a complement, denoted
$Y_{1}$, of $\ker F_{22}$ in $Y$, i.e., $Y$ is topologically isomorphic
to $\ker F_{22}\oplus Y_{1}$.

We claim that $\inf\set{\norm{F_{22}y}}{y\in Y_{1},\ \norm y=1}>0$.
Suppose to the contrary that there exists a sequence $\{y_{n}\}\subseteq Y_{1}$,
with $\norm{y_{n}}=1$ for all $n\in\N$, such that $F_{22}y_{n}\to0$.
Then $y_{n}-E_{21}SF_{12}y_{n}=E_{22}F_{22}y_{n}\to0$ as $n\to\infty$.
Since $S$ is compact and $\norm{y_{n}}=1$ for all $n\in\N$, there
exists a subsequence $\{y_{n_{k}}\}$ of $\{y_{n}\}$ such that $\{E_{21}SF_{12}y_{n_{k}}\}$
converges, with limit denoted $y$. But then $y_{n_{k}}=E_{22}F_{22}y_{n_{k}}+E_{21}SF_{12}y_{n_{k}}\to0+y$
as $k\to\infty.$ Since $\{y_{n}\}\subseteq Y_{1}$, with $\norm{y_{n}}=1$
for all $n\in\N$, we obtain $y\in Y_{1}$ and $\norm y=1$. Since
$y\in Y_{1}$ and $Y_{1}$ is a complement of $\ker F_{22}$, we obtain
$0\neq F_{22}y=\lim_{k\to\infty}F_{22}y_{n_{k}}=0$, which is absurd.
We conclude that $\inf\set{\norm{F_{22}y}}{y\in Y_{1},\ \norm y=1}>0.$

Now defining $X_{1}:=\range\left(F_{22}|_{Y_{1}}\right)$, the operator
$F_{22}|_{Y_{1}}:Y_{1}\to X_{1}$ is bijective with bounded inverse.
Therefore $X_{1}$ is complete, and hence closed. The operator $F_{22}|_{Y_{1}}$
is then the sought topological isomorphism. 
\end{proof}
We recall that a Banach space $X$ is called \textit{prime} \cite[Definition 2.2.5]{AlbiacKalton},
if every infinite dimensional complemented subspace of $X$ is topologically
isomorphic to $X$. Standard examples of prime spaces are the space
of convergent sequences $c$, the space of sequences converging to
zero $c_{0}$ (both endowed with the uniform norm), and $\ell^{p}$
with $1\leq p\leq\infty$ (cf.~\cite{AlbiacKalton}). The following
corollary follows immediately from the previous result: 
\begin{cor}
Let $X$ and $Y$ be Banach spaces with $Y$ prime. If $S\in B(Y)$
is compact and equivalent after extension to some $T\in B(X)$, then
$X$ contains a copy of $Y$. 
\end{cor}
Theorem \ref{prop:Ran_finite_codim_subspace_embeds} shows that any
Banach space property that $X$ may have, that is also transferred
to its closed subspaces and also preserved under the taking of direct
sums with finite dimensional spaces, must transfer to $Y$. We make
this precise by stating the following definition and subsequent result. 
\begin{defn}
Let $X,Y$ and $Z$ be a Banach spaces and let $(P)$ be a Banach
space property. 
\begin{enumerate}
\item We will say $X$ \emph{transfers property $(P)$ to closed subspaces},
if every closed subspace of $X$ has property $(P)$. 
\item We will say $(P)$ is \emph{preserved under direct sums} \emph{with
finite dimensional spaces}, if $Y\oplus Z$ has property $(P)$, whenever
$Y$ has property $(P)$ and $Z$ is finite dimensional. 
\end{enumerate}
\end{defn}
\begin{prop}
\label{prop:eae-to-compact-implies-transferrence-of-properties}Let
$(P)$ be a Banach space property, and let $X$ and $Y$ be Banach
spaces with $T\in B(X)$ and $S\in B(Y)$ equivalent after extension.
If $S$ is compact, $X$ transfers property $(P)$ to closed subspaces,
and $(P)$ is preserved under direct sums with finite dimensional
spaces, then $Y$ has property $(P)$.
\end{prop}
\begin{proof}
By Theorem \ref{prop:Ran_finite_codim_subspace_embeds}, there exists
a complemented subspace $Y_{1}$ of $Y$ with finite codimension that
is topologically isomorphic to a closed subspace of $X$. Since $X$
transfers property $(P)$ to closed subspaces, $Y_{1}$ has property
$(P)$. Also, $(P)$ is preserved under direct sums with finite dimensional
spaces, so $Y$ has property $(P)$, because $Y$ is topologically
isomorphic to a direct sum of $Y_{1}$ and a finite dimensional complement
of $Y_{1}$. 
\end{proof}
We briefly demonstrate some applications of the previous proposition
in the next corollary.

We refer the reader to \cite{DiestelUhl} for definitions of the Radon-Nikodym
property and the Dunford-Pettis property. A Banach space is said to
have the \emph{hereditary Dunford-Pettis property }if each of its
closed subspaces has the Dunford-Pettis property. 
\begin{cor}
\label{cor:properties-that-transfer}Let $X$ and $Y$ be Banach spaces
with $T\in B(X)$ and $S\in B(Y)$ equivalent after extension. If
$S$ is compact, then: 
\begin{enumerate}
\item If $X$ is isomorphic a Hilbert space, then so is $Y$. 
\item If $X$ is separable, then so is $Y$. 
\item If $X$ is reflexive, then so is $Y$. 
\item If $X$ has the Radon-Nikodym property, then so does $Y$. 
\item If $X$ has the hereditary Dunford-Pettis property, then so does $Y$. 
\end{enumerate}
\end{cor}
\begin{proof}
The results all follow from Proposition \ref{prop:eae-to-compact-implies-transferrence-of-properties}:

If $X$ is respectively isomorphic to a Hilbert space, separable or
has the hereditary Dunford-Pettis property, then straightforward arguments
will show that each of these properties is respectively transferred
to closed subspaces and preserved under taking direct sums with finite
dimensional spaces. This establishes (1), (2) and (5).

Every closed subspace of a reflexive space is reflexive \cite[Theorem 1.11.16]{Megginson},
and an elementary argument will establish that direct sums of reflexive
spaces with finite dimensional spaces are reflexive, establishing
(3).

Every Banach space with the Radon-Nikodym property transfers the Radon-Nikodym
property to its closed subspaces \cite[Theorem III.3.2]{DiestelUhl},
and an elementary argument will establish that the direct sum of a
Banach space with the Radon-Nikodym property and a finite dimensional
space has the Radon-Nikodym property, establishing (4). 
\end{proof}

\bibliographystyle{amsplain}
\bibliography{bib}

\end{document}